\documentclass[11pt,a4paper]{amsart}

\usepackage{amssymb}

\renewcommand{\epsilon}{\varepsilon}
\renewcommand{\setminus}{\smallsetminus}
\renewcommand{\emptyset}{\varnothing}

\newtheorem{theorem}{Theorem}[section]
\newtheorem{proposition}[theorem]{Proposition}
\newtheorem{corollary}[theorem]{Corollary}
\newtheorem{lemma}[theorem]{Lemma}

\theoremstyle{definition}

\theoremstyle{remark}

\newcommand{\Q}{\mathbb Q}
\newcommand{\Z}{\mathbb Z}

\newcommand{\R}{\mathbb R}

\newcommand{\BR}{\mathbb{R}}

\newcommand{\BZ}{\mathbb{Z}}
\newcommand{\BN}{\mathbb{N}}

\newcommand{\fpinfty}{{\FP}_{\infty}}

\newcommand{\FFF}{\operatorname{F}}
\newcommand{\FP}{\operatorname{FP}}

\newcommand{\cohom}[3]{H^{{\raise1pt\hbox{$\scriptstyle#1$}}}(#2\>\!,#3)}
\newcommand{\tatecohom}[3]%
     {\widehat H^{{\raise1pt\hbox{$\scriptstyle#1$}}}(#2\>\!,#3)}

\newcommand{\Cohom}[3]%
     {H^{{\raise1pt\hbox{$\scriptstyle#1$}}}\big(#2\>\!,#3\big)}
\newcommand{\Tatecohom}[3]%
     {\widehat H^{{\raise1pt\hbox{$\scriptstyle#1$}}}\big(#2\>\!,#3\big)}

\newcommand{\homol}[3]{H_{{\lower1pt\hbox{$\scriptstyle#1$}}}(#2\>\!,#3)}
\newcommand{\homolog}[2]{H_{{\lower1pt\hbox{$\scriptstyle#1$}}}(#2)}

\newcommand{\mono}{\rightarrowtail}
\newcommand{\epi}{\twoheadrightarrow}

\newcommand{\eg}{{\underline EG}}

\title[finite groups acting on Thompson groups]{Fixed points of finite
groups acting on generalised Thompson
groups}
\author{D.~ H. ~Kochloukova}
\address{Dessislava H.~Kochloukova, Department of Mathematics,
University of Campinas,
Cx. P. 6065,
13083-970 Campinas, SP, Brazil}
\email{desi@unicamp.br}
\author{C.~Mart\'inez-P\'erez}
\address{Conchita Mart\'inez-P\'erez, Departamento de Matem\'aticas,
Universidad de
Zaragoza,
50009 Zaragoza, Spain} \email{conmar@unizar.es}
\author{B.~ E.~A.~Nucinkis}
\address{Brita E.~A.~Nucinkis, School of Mathematics, University of
Southampton,
Southampton,
SO17 1BJ, United Kingdom}
\email{bean@soton.ac.uk}
\date{\today} 

\thanks{The first author is partially supported by "bolsa de produtividade
em pesquisa, CNPq". The second author is partially supported by Gobierno de Aragon and MTM2007-68010-C03-01. This work was also partially supported by Royal Society International Travel Grant TG08182 and LMS Scheme 4 Grant 4814.}

\begin{document}

\thispagestyle{empty}

\begin{abstract} We study centralisers of finite order automorphisms
of the generalised
Thompson groups $F_{n, \infty}$ and conjugacy classes of  finite
subgroups in finite
extensions of $F_{n, \infty}$.  In particular we show that centralisers of
finite automorphisms in $F_{n,\infty}$ are either of type $\FP_\infty$ or
not finitely generated.

\noindent As an application we deduce the following
result about the
Bredon type of  such finite extensions: any finite extension  of $F_{n,
\infty},$ where the elements of
finite order act on $F_{n, \infty}$ via conjugation with piecewise-linear
homeomorphisms,
is of type Bredon $\FFF_{\infty}$.
In particular finite extensions of $F = F_{2, \infty}$ are  of type
Bredon $\FFF_{\infty}$.  \end{abstract}
\maketitle


\section{Introduction}

In the 1960s R. Thompson defined the group $F$, which has a realisation as
a subgroup of the group of increasing PL (piece-wise linear)
homeomorphisms of the unit interval. For a fixed natural number $n$ the
group $F$ has a generalization  $F_{n}$  given by the infinite
presentation
$$\langle x_0, x_1, \ldots , x_i, \ldots | x_i^{x_j} = x_{i+ n - 1} \hbox{
for } 1\leq i,  0 \leq j < i \rangle.$$
Furthermore the group $F_{n}$ has a realization, denoted $F_{n,\infty}$,
as a subgroup  of the group of increasing PL homeomorphisms of the real
line $\BR$, which we consider in this paper, see \cite[Section 2]{BrinGuzman}.

The R. Thompson group $F$ satisfies strong cohomological finiteness
conditions: it is of homological type $\FP_\infty$, is finitely
presented \cite{BrownRoss} but has infinite cohomological dimension. The
same properties hold for the generalised R. Thompson groups  of \cite{Brown2}
and some further generalizations of $F$ as in \cite{Stein}.

The homological property $\FP_{\infty}$ has a homotopical counterpart.
The coresponding homotopical condition $\FFF_{\infty}$ requires that
the group admits a finite type model for $EG$, the universal cover of the
Eilenberg-Mac Lane space.
A group of type $\FP_{\infty}$ is not necessary of type $\FFF_{\infty}$
\cite{BB} and in general
a group $G$ is of type $\FFF_\infty$ if and only if it is finitely
presented and of type $\FP_{\infty}$. Though the group
$F_{n,\infty}$ admits a finite type model for $EG$, such a model is never
finite dimensional as $F_{n, \infty}$ has infinite cohomological
dimension.

In this paper we consider Bredon cohomological finiteness conditions of
finite extensions of generalised Thomspon groups. A Bredon cohomological
finiteness condition is a finiteness condition satisfied by groups
admitting a cocompact model for $\eg$, the classifying space for proper
actions.  A $G$-CW-complex $X$ is a model for $\eg$, if $X^H$ is
contractible whenever $H\leq G$ is finite and $X^H = \emptyset$ otherwise.
By definition  a group $G$ is of type Bredon $\FFF_\infty$ if it admits a
finite type model for $\eg$. For background on  Bredon type $\FFF_\infty$
the reader is referred to \cite{Luck, Mislinsurvey}.  Obviously if the
group $G$ is torsion-free, as are the generalised Thompson groups
considered here,  the  Bredon type $\FFF_\infty$ is equivalent to the
ordinary homotopical type $\FFF_{\infty}$.

In \cite{Luck}, W. L\"uck gave an algebraic criterion for an arbitrary
group to admit a finite type model for $\eg.$ In particular, the following
two conditions are equivalent:
\begin{itemize}
\item[(i)] $G$ is of type Bredon $\FFF_\infty$
\item[(ii)] $G$ has finitely many conjugacy classes of finite subgroups
and centralisers $C_G(H)$ of finite subgroups $H$ of $G$ are of type
$\FFF_\infty.$
\end{itemize}

Although the ordinary finiteness conditions $\FP_\infty$ and $\FFF_\infty$
are preserved under finite extensions, it was shown  in
\cite{LearyNucinkis} that generally this does not hold in the Bredon
context. In particular, there are examples of finite extensions of groups
of type $\FFF$, i.e. of groups admitting a finite Eilenberg-Mac Lane
space, for which either of the above conditions fails.  For large classes of
groups, finite extensions of groups of type Bredon $\FFF_\infty$ are again
of type Bredon $\FFF_\infty$. Examples include hyperbolic groups \cite{MS}
and soluble groups \cite{MartinezNucinkis, KMN2}.

In this paper we study finite extensions of $F_{n, \infty}$.  L\"uck's
result implies that in order to study the Bredon type of finite extensions of $F_{n,
\infty}$, we first have to study  the group of fixed points of finite order
automorphisms of $F_{n, \infty}$. Our first main result points exactly in this
direction.

\medskip\noindent
{\bf Theorem A.} \label{thmA} {\it Let $F_{n, \infty}$ be a generalised
Thompson group and
$\varphi$ be an automorphism of finite order of  $F_{n, \infty}$. Then
\begin{itemize}
\item[(1)]There is a decreasing homeomorphism $h$ of the real line such
that
$\varphi(f) =
h^{-1} f h$, $h^2 = id$ and there is a unique real number $t_0$ such
that $(t_0)h = t_0$;
\item[(2)]The group of the  fixed points of $\varphi$ is of type
$\FFF_{\infty}$
if and only if
there is an $f \in F_{n, \infty}$ such that $(t_0)f = t_0$ and the right-hand
slope of $f$ at
$t_0$ is not 1.
\item[(3)] If $h$ is PL then the group of the  fixed points of $\varphi$
is of
type $\FFF_{\infty}$.
In particular, if $n = 2$ the group of the fixed points of $\varphi$ is of
type
$\FFF_{\infty}$.
\end{itemize}}

\medskip\noindent Part (2) of Theorem A is proved in Sections
\ref{sectiondirection} and \ref{theotherdirection} and  the proof  is
completed in Section \ref{sectionproofA}. The last line of part (3) of
Theorem A  follows directly from the fact that the structure of $Aut(F)$
is well understood, see \cite{Brin}. In particular,  any automorphism of
$F$ is given by conjugation with a PL homeomorphism $h$. The case $n \geq
3$ is much more complicated:  there are automorphisms of
$F_{n, \infty}$  given by conjugation with homeomorphisms of the
real line $h$ which  are not PL \cite{BrinGuzman}. Such automorphisms are
called exotic. In the final Section 10  we exhibit infinitely many exotic automorphisms of
finite order for any $n \geq 3$, making substantial use of a construction
from \cite{BrinGuzman}.

\noindent
Note that the proof of Lemma \ref{swaping0}
relies on a Bieri-Strebel-Neumann result about the $\Sigma^1$-invariant of
finitely generated subgroups of groups of PL homeomorphisms of a closed
interval \cite{BNS}. This together with Lemma \ref{swaping} is one crucial ingredient of the proof of Theorem A. In the preliminaries section we include the necessary background
for $\Sigma$-theory  required later on. It is 
interesting to note that methods from $\Sigma$-theory were also involved when finiteness conditions of centralisers of finite subgroups of soluble groups where considered  in \cite{MartinezNucinkis}.

The following proposition can be deduced as a corollary of Theorem
A, but it has an independent short proof, see Corollary \ref{propB}. Note
that we do
not know of an example where the group of fixed points of $\varphi$ is
infinitely generated.

\medskip\noindent
  {\bf Proposition B.} {\it Let $F_{n, \infty}$ be a generalised
Thompson group and let $\varphi$ be an automorphism of finite order of
$F_{n, \infty}$. Then  the group of  fixed points of $\varphi$ is either
infinitely generated or of type $\FFF_{\infty}$.}

\medskip\noindent
The proof of Proposition B together with Theorem A imply that in
the second part of Theorem A  the condition $f \in F_{n, \infty}$ can be
substituted with $f \in
C_{F_{n,\infty}}(\varphi)$, see Corollary \ref{equivalent}.

The next step in our programme on  Bredon finiteness properties is to study
the conjugacy classes of finite subgroups in a finite extension of $F_{n,
\infty}$.
By the first part of Theorem A we obtain that  finite order non-trivial
automorphisms of $F_{n, \infty}$ have order 2.    This motivates the study
 of  conjugacy classes of
finite subgroups of split extensions of $F_{n, \infty}$ by the cyclic
group of order 2. The reader is referred to Section \ref{sectionconjugacy} and
 Theorem \ref{thmclasses} in particular. Combining Theorem \ref{thmclasses} with
Theorem A
and L\"uck's  criterion,  we obtain the main result of this
paper :

\medskip\noindent
     {\bf Theorem C.} {\it Any finite extension $G$ of $F_{n, \infty}$
where the elements of
finite order act on $F_{n, \infty}$ via conjugation with PL
homeomorphisms  is Bredon
$\FFF_{\infty}$.}

\medskip\noindent  We actually prove a slightly stronger version of
Theorem C. In Lemma \ref{order2} we show that any non-identity finite
order automorphism of $F_{n,\infty}$ is given by conjugation with  a
homeomorphism $h$ of the real line fixing a unique $t_0 \in \R.$ The proof
of Theorem C only requires that $t_0 \in \frac 12 \Z[\frac 1n]$ and this
is satisfied for piecewise linear $h$.   As a corollary of Theorem C we
obtain the following result.

  \medskip\noindent
{\bf Corollary D.} {\it Any finite extension of $F$ is Bredon
$\FFF_{\infty}$.}

  \medskip\noindent
     The proofs of Theorem C and Corollary D are completed in Section
\ref{sectionproofCD}.

\section{Preliminaries on the generalised R. Thompson groups}

\noindent Adopting the notation of \cite{BrinGuzman}, we write  $F_{n,
\infty}$ for
  the group of PL increasing homeomorphisms $f$ of $\BR$ acting on the
right on the real line such that the set $X_f$ of break points of $f$ is a
discrete subset of $\BZ[{1 \over n}]$,
$f(X_f) \subseteq
\BZ[\frac{1}{n}]$ and slopes are integral powers of $n$. Furhtermore,
there
are integers $i$ and
$j$
(depending on $f$) with
$$(x)f = \begin{cases} x + i (n-1)  \mbox{ for } x > M,  \\
        x + j (n-1) \mbox{ for }  x < - M \end{cases}$$
  for sufficiently large $M$ (depending on $f$ again). For
$t \in \BZ$  we define  $F_{n,t}$ to be the subgroup of $F_{n, \infty}$ of
all
elements which are the
identity map on the interval $(- \infty, t]$. By
\cite[Lemma~2.1.2]{BrinGuzman}
$F_{n,0}$ is isomorphic to the generalised Thompson group $F_{n}$.
Furthermore,  by
\cite[Lemma~2.1.6]{BrinGuzman} $F_{n,
\infty} \simeq F_{n,t}$ for every $t \in \BZ$.  Similarly, for
$t \in \BR$ we define
$$F_{n,t} = \{ f \in F_{n, \infty} \mid \hbox{ the restriction of } f
\hbox{ on } (-\infty, t] \hbox{ is the identity map} \}$$

\begin{lemma} \label{isomorphism} If $t \in \BZ[ \frac{1}{n}]$ then
$F_{n,t} \simeq
F_{n,0}$. In particular $F_{n,t}$ is of type $\FFF_{\infty}$.
\end{lemma}

\begin{proof}  The map $\theta : F_{n, t} \to F_{n, 0}$ defined by
$\theta(f) = g f
g^{-1}$, where $(x)g = x + t$, is an isomorphism.
\end{proof}

\section{Preliminaries on Sigma theory}

\noindent In this section we collect facts about the homological $\Sigma$
invariants, which we shall use in Section  \ref{sectiondirection}. Note
that there is also a
homotopical version, but we will not make any use of it here.  Let $G$ be
a
finitely generated group and  $\chi : G \to \BR$  be a non-zero character.
We define the monoid
$$
G_{\chi} = \{ g \in G \mid \chi(g) \geq 0 \}
$$
and the character sphere
$$S(G) = Hom(G, \BR) \setminus \{ 0 \} /\sim,$$
  where $\sim$ is an equivalence relation on $Hom(G, \BR) \setminus \{ 0
\}$ defined in the following way : $\chi_1
\sim \chi_2$ if and only if there is a positive real number $r$ such
that $r \chi_1 =
\chi_2$.
By definition, see \cite{BieriRenz},  for the trivial right $\BZ G
$-module $\BZ$
$$
\Sigma^m(G, \BZ) = \{ [\chi] \in S(G) \mid \BZ \hbox{ is } \FP_m \hbox{
over } \BZ
G_{\chi} \},
$$
where $[\chi]$ denotes the equivalence class of $\chi$ and $m\geq 0$
or $m=\infty$.

\subsection{Meinert's results}
The following is a result due to H. Meinert, but the proof was
published  in
\cite[Lemma~9.1]{Gerhke}.

\begin{lemma} \label{meinert1} Let $G_1$ and $G_2$ be groups of type
$\FP_{d+1}$, and let $i$ and
$j$ be non-negative integers such that $ i + j \leq d$. Then
$$
\pi_1^*( \Sigma^i(G_1, \BZ)) + \pi_2^* (\Sigma^j(G_2, \BZ)) \subseteq
\Sigma^{i+j +
1}(G_1 \times G_2, \BZ)
$$
where $\pi_i^* : S(G_i) \to S(G_1 \times G_2)$ is  the map induced by
the canonical
projection $\pi_i : G_1 \times G_2 \to G_i $.
\end{lemma}

\begin{corollary} \label{meinert2} Let $G_1$ and $G_2$ be groups of
type $\FP_{\infty}$, and for $i\in \{1,2\}$ let
$\chi_i \in Hom(G_i, \BR) \setminus \{ 0 \}$.
Furthermore suppose
  that $[\chi_i]\in\Sigma^{\infty} (G_i, \BZ)$ for at least one $i\in
\{1,2\}.$

\noindent
Then for $G = G_1 \times G_2$ and
$\chi = (\chi_1,\chi_2) \in Hom (G, \BR)\setminus\{0\}$ we have that $[\chi]
\in \Sigma^{\infty}(G, \BZ)$.
\end{corollary}

\begin{proposition} \label{meinertnew} \cite[Prop.~4.2]{Meinert1},
\cite[Thm.~4.3]{Meinert2} Let $H$ be a group of type $\FFF_m$ and $\varphi
\in End(H).$ Put $G = \langle H , t \mid t^{-1} H t =
\varphi(H) \rangle.$
Suppose  $\chi : G \to \BR$ is a real character such that $\chi(H) = 0$
and $\chi(t) = 1$. Then $-
[\chi] \in \Sigma^{m}(G, \BZ)$.
\end{proposition}

\subsection{$\Sigma^1$ for PL groups}

    Let $G$ be a finitely generated group of PL orientation preserving
homeomorphisms of an
interval $I = [y_1, y_2]\subset \R$  acting on the right.  Let
$\chi_1, \chi_2 : G
\to \BR$ be the characters given by $\chi_i(f) = log_m(y_i)f'$, where
$m>1$ is a
fixed  number.

\begin{theorem} \cite[Thm.~8.1]{BNS} \label{sigma1}  Assume that $G$ does
not
fix any point in the open interval $(y_1, y_2)$. Further assume that
$\chi_1 (Ker (\chi_2)) = Im (\chi_1)$ and
$\chi_2 (Ker(\chi_1)) = Im (\chi_2).$ Then
    $$
    \Sigma^1(G,\BZ) = S(G) \setminus \{ [\chi_1], [\chi_2] \}.
    $$
    \end{theorem}

\noindent
Except of the case $G = F = F_{2,\infty}$ \cite{BGK}, there are, untill now, no
known results about the higher dimensional
$\Sigma$-invariants  of groups $G$ of PL orientation preserving
homeomorphisms of closed intervals.

\subsection{Direct products formula in dimension 1}

In small dimensions, the homological $\Sigma$-invariants of direct product
of groups can easily  be calculated via the $\Sigma$-invariants of the
direct factors. There is a
version  of the direct product formula   for $\Sigma$-invariants with
coefficients in
a field  without  restriction of the dimension \cite{BG}. But we will not
make use of this result.

\noindent
For a group $G$ denote
$$\Sigma^1(G, \BZ)^c=S(G)\setminus\Sigma^1(G,\BZ).$$
The following proposition is the direct product formula in dimension 1.

\begin{proposition} \label{direct1}
$$\begin{aligned}
\Sigma^1(G_1& \times G_2, \BZ)^c
=\\
& \{ [( \chi_1,\chi_2)] \mid \hbox{ one of } \chi_1, \chi_2  \hbox{ is
zero and for } \chi_i \not= 0,
[\chi_i] \in \Sigma^1(G_i, \BZ)^c \}.\\
\end{aligned}$$
\end{proposition}
\begin{proof} We apply  Lemma \ref{meinert1}  to $i = j
= 0$ to get
$$\Sigma^1(G_1\times G_2,\BZ)^c\subseteq\{ [( \chi_1,
\chi_2)] \mid \hbox{ one of } \chi_1, \chi_2  \hbox{ is zero} \}.$$
The result follows easily from the definition of  $\Sigma^1$
taking also into account that the property  $\FP_1$ behaves well for
quotients.
\end{proof}

\section{Fixed points of automorphisms of finite order
}\label{sectiondirection}

\begin{lemma} \label{order2} Every non-identity finite order
automorphism $\varphi$  of
$F_{n, \infty}$  has order 2. Furthermore $\varphi(f) = h^{-1} f h$
for a unique
homeomorphism $h$ of the real line $\BR$, $h^2$ is the identity map
and $h$ is a
decreasing function, so there is an unique $t_0 \in \BR$ such that
$(t_0)h = t_0$.
\end{lemma}

\begin{proof}
By \cite{McClearyRubin} and \cite[Lemma~6.1.1,~Lemma~6.1.2]{BrinGuzman},
for any $\varphi \in
Aut(F_{n,\infty})$ there is a unique homeomorphim $h$ of $\BR$ such
that \begin{equation}
\label{unique} \varphi (f) = h^{-1} f h.\end{equation}  In the case
when $n = 2$ it was
shown in \cite[Lemma~5.1]{Brin} that $h$ is piecewise linear but this
does not hold for
$n \geq 3$ \cite[Thm.~6.1.5]{BrinGuzman}. Now assume that $\varphi$
has  finite order,
say $k$. Then by the uniqueness of the conjugating homeomorphism  in
(\ref{unique}) we
deduce that $h^k = id$.

We remind the reader that, following the notation of \cite{BrinGuzman},
all
homeomorphisms of the real line act on the right.
For the   homeomorphism $h$ there are just two possibilities:
increasing or decreasing.
In the first case consider $x_0$ such that $x_0 \not= (x_0)h$.   If
$x_0 < (x_0)h$ since
$h$ is increasing
$x_0 < (x_0)h < (x_0) h^2 < \ldots <  (x_0) h^k$, giving a contradiction. If $x_0
>(x_0) h$, again using  that $h$ is increasing, $x_0 > (x_0) h > (x_0) h^2 >
\ldots > (x_0) h^k$, giving a contradiction as well. Then $h$ is decreasing and so
the function $(x)g =
(x)h - x$ is decreasing and there is an unique $t_0 \in \BR$ such that
$(t_0) g = 0$.
Finally, observe that $h^2$ is increasing and of finite order and
therefore $h^2=1$.
\end{proof}

\noindent
Using the notation of Lemma \ref{order2} we set
$$C_{F_{n, \infty}} (\varphi) = \{ f \in F_{n, \infty} \mid \varphi(f)
= f \}.$$

\begin{lemma}\label{fixed} $C_{F_{n, \infty}}(\varphi)$ fixes the point
$t_0$.
\end{lemma}

\begin{proof} If $f \in C_{F_{n, \infty}}(\varphi)$ we have $h f = f h$.
So $(t_0)f h = (t_0) h f = (t_0) f$. Since by Lemma \ref{order2}  $t_0$ is
the unique fixed
point of $h$, we deduce that $(t_0) f = t_0$.
\end{proof}

\begin{proposition} \label{reduction1} If $ t_0 \in \BZ[{1 \over n}]$
then $C_{F_{n,\infty}}(\varphi)$ is of type $\FFF_{\infty}$.
\end{proposition}

\begin{proof}
We claim that
$$  C_{F_{n, \infty}}(\varphi) \cong F_{n, t_0}.$$
Consider the map $\Theta: C_{F_{n, \infty}}(\varphi) \to F_{n, t_0}$
defined by $\Theta(f) =\tilde f$, where
$\tilde{f}$ coincides
with $f$ on the interval $[t_0, \infty)$ and is the  identity on
$(-\infty, t_0]$.  The condition $fh = hf$ implies that $\Theta$ is
injective. Given $f \in F_{n,t_0}$, we can find an element $\hat f \in
F_{n,\infty}$ satisfying $\varphi(\hat f)=h\hat fh=\hat f$ and
$\Theta(\hat f) =f$ as follows:
$$(t)\hat f =\begin{cases} (t)f \mbox{ for } t>t_0 \\
(t)hfh \mbox{ for } t\leq t_0. \end{cases} $$
Here we have used  that $h$ is decreasing of order 2 and is fixing $t_0.$

\noindent To finish the proof we use that Lemma
\ref{isomorphism}  shows that $F_{n,t_0}$ is of type $\FFF_{\infty}$.
\end{proof}

\medskip\noindent
In the rest of this section we complete the proof  of one of the
directions of Theorem
A, part (2), see Corollary \ref{importantcor}.  More precisely, we show
that if there is some $f\in F_{n,\infty}$ such that $(t_0)f=t_0$ and
the right-hand slope of $f$ at $t_0$ is not 1, then
$C_{F_{n,\infty}}(\varphi)$ is of type
$\FP_\infty$. By Proposition \ref{reduction1} we can assume that
$$
t_0 \notin \BZ[{1 \over n}].
$$
Note that  \cite[Lemma~5.1]{Brin} implies that for $n = 2$ the homeomorphism  $h$
is PL. Together with some technical results proved independently later on, see the first part of Theorem \ref{thmclasses}, this  implies that for $n
=2$ we have  $t_0 \in
\BZ[{1 \over n}].$
In what follows we can therefore assume  that
$$
n \geq 3.
$$

\medskip\noindent
{\bf Definition.} Let $I = [\alpha, \beta]$ be an interval in $\BR$. By
some abuse of notation we also include the cases $\alpha=-\infty$ and
$\beta =\infty.$
We set
$$
F_{n,I} = \{ f \mid  f \hbox{ is the restriction to }I  \hbox{ of }
\tilde{f} \in F_{n, \infty}, \tilde{f}(\alpha) =\alpha,\tilde{f}(\beta) =
\beta\}.$$

\noindent In this section we consider the following condition:
\begin{equation}  \begin{array}{c} \hbox{ There is an element }\hat{f} \in
F_{n, \infty}, \hbox{such that }(t_0)\hat{f} = t_0  \\ \hbox{ and the
slope of } \hat{f} \hbox{ at }t_0 \hbox{  is not 1.}\end{array}
\label{condition}
\end{equation}

\noindent Moreover by considering either $\hat{f}$ or $\hat{f}^{-1},$
we may assume $log_n((t_0)\hat{f}') < 0$ and  that $ \mid
log_n((t_0)\hat{f}' )\mid$ is as small as possible, that is, the  slope at
$t_0$ is
smaller than 1 but as close as possible to 1. Throughout this section we
write $\hat f$ for this element.

\noindent
We also consider the characters $\mu_1, \mu_2 : F_{n, [t_0, \infty]}
\to \BR$ given by
$$\begin{array}{ll} \mu_1(f) = log_n
((t_0) f') &\hbox{ and } \\ \mu_2(f) = - i,& \end{array} $$ where $(x) f = x
+ i (n-1)$
for sufficiently large $x$.

\begin{proposition} \label{HNNext} Suppose that (\ref{condition}) holds.
Then the group $F_{n, [t_0, \infty]}$ is an ascending HNN extension
with a base group
isomorphic to $F_{n,\infty}$. In particular, $F_{n, [t_0, \infty]}$ is of
type
$\FFF_{\infty}$.
Similarly $F_{n, [ - \infty, t_0]}$ is of type $\FFF_{\infty}$.
\end{proposition}

\begin{proof}
Define $M$ to be the subgroup of $F_{n, [t_0, \infty]}$ containing those
elements with slope 1 at $t_0$, i.e., $M=\text{Ker} (\mu_1)$. Set $\widetilde f$ to
be the restriction
of $\hat f$ to $[t_0,\infty]$. Note that the fact that $\mu_1$ is a
discrete character
and the choice of $\widetilde f$ imply that $F_{n,
[t_0,\infty]}= \langle M,\widetilde f \rangle$.

Let $r$ be an element of $\BZ[{1 \over n}]$ such that $r > t_0$. We
identify $F_{n, r}$
    with its restriction on the interval $[t_0, \infty)$ i.e. we
identify $F_{n,r}$ with its
image
in $F_{n, [t_0, \infty]}$. Thus if $r_1 > r_2 > \ldots > r_i > \ldots
$ is a decreasing
sequence of elements of $\BZ[{1 \over n}]$ with limit $t_0$ we get that
$M = \cup_i F_{n, r_i}$.

Note that for $r$ as in the previous paragraph, with $r$ sufficiently
close to $t_0$, the
restriction of $\widetilde f$ on the interval $[t_0, r]$ is linear. 
Since the slope of
$\widetilde f$ at $t_0$ is smaller that 1, we have $t_0 = (t_0) \widetilde{f} <
(r)\widetilde f <  r$. Hence
\begin{equation} \label{99}
F_{n,r}^{\widetilde f} = \widetilde{f}^{-1} F_{n,r} {\widetilde f} = F_{n,
(r) \widetilde f} \supseteq F_{n, r},
\end{equation}
and also for $k \in \BN$
$$
F_{n,r}^{\widetilde f^{k}} = F_{n, (r)\widetilde f^k} \supseteq
F_{n,r}^{\widetilde f^{k-1}} = F_{n, (r)\widetilde f^{k-1}}\supseteq
\ldots
$$
Define $r_k = (r)\widetilde f^k$. Then  the sequence $\{ r_k \}$
satisfies the assumptions of the previous paragraph. Thus
$$
M = \cup_{k \geq 1} F_{n,r}^{\widetilde f^{k}}
$$
and by (\ref{99}) $F_{n, [t_0, \infty]}$ is an ascending HNN
extension with a base group
$F_{n,r}$ and stable letter $\widetilde f$. By Lemma \ref{isomorphism}
$F_{n,r}\cong
F_{n,\infty}$ which
is of type $\FFF_{\infty}$.
\end{proof}

\begin{lemma} \label{nonfix} The group $F_{n, [t_0, \infty]}$ does not
fix any element of
$(t_0, \infty)$.
\end{lemma}

\begin{proof} Let $t_1 \in (t_0, \infty)$ and $r \in \BZ[{1 \over n}]
\cap (t_0, t_1)$.
As $F_{n,r}$ does not fix an element of $(r, \infty)$ and $F_{n,r}$
embeds in  $F_{n,
[t_0, \infty]},$ we  deduce that $F_{n, [t_0, \infty]}$ does not fix
$t_1$.
\end{proof}

\begin{lemma} \label{independent}
The characters $\mu_1, \mu_2 : F_{n, [t_0, \infty]} \to \BR$ satisfy
    $\mu_1(Ker (\mu_2)) = Im (\mu_1)$ and $\mu_2( Ker
(\mu_1)) = Im (\mu_2)$.
\end{lemma}
\begin{proof}
Again, let $\widetilde{f}$ be the restriction of $\hat{f}$ to $[t_0,
\infty)$. Choose for any $i\in\BZ$ an
$\hat{f}_i
\in F_{n, \infty}$ such that $
(x) \hat{f}_i = x + i (n-1)$ for sufficiently large $x$ and $(t_0)
\hat{f}_i = t_0$. This
is possible, as we can choose $\hat{f}_i$ to be the identity map on an
interval $I_i$
containing $t_0$  with end points in $\BZ[{1 \over n } ]$.
Denote by $f_i$ the
restriction of $\hat{f}_i$ to $[t_0, \infty]$. Then $\mu_1(f_1) = 0,$
$\mu_2(f_1) = -1$, i.e., $f_1\in\text{Ker} (\mu_1)$ and $\mu_2(f_1)$
generates
$\text{Im} (\mu_2)$. Moreover,
for any $i\in\BZ$
$$\mu_1(\widetilde{f} f_i) =
\mu_1(\widetilde{f}) + \mu_1({f}_i) = \mu_1(\widetilde{f})$$
and $$\mu_2(\widetilde{f} f_i) = \mu_2(\widetilde{f}) + \mu_2({f}_i) =
\mu_2(\widetilde{f}) - i.$$
So we may choose some $i$ with $\mu_2(\widetilde{f} f_i) = 0$,
but $\mu_1(\widetilde{f} f_i) = \mu_1(\widetilde{f})$ is a generator of
$\text{Im}(\mu_1)$ by the choice of
$\hat f$.
\end{proof}

\noindent
Without loss of generality  we can assume that $t_0 \in (0, 1)$.
Indeed, otherwise we can
interchange $h$ with another homeomorphism $g^{-1} h g$, where $(x) g = x
+
i$ for some $i \in \BZ$ to reduce to $t_0 \in [0, 1)$.  Since $t_0 \notin
\BZ [{1 \over n}],$ we have $t_0 \not= 0$.

\begin{lemma} \label{interval} The group $F_{n, [t_0, \infty]}$ has a
realization as a
subgroup of the group $H$ of  piecewise linear transformations of the
interval
$[t_0,n-1]$. 

\noindent
Furthermore, the characters $\chi_1$ and $\chi_2$ of $H$
given by $log_n$ of
derivatives at $t_0$ and $n-1$ correspond  to the characters $\mu_1,
\mu_2 : F_{n, [t_0,
\infty]} \to \BR$ given by $\mu_1(f) = log_n ((t_0)f')$ and $\mu_2(f)
= - i$ if at
infinity $(x)f = x + i (n-1)$.
\end{lemma}

\begin{proof}

By \cite[Lemma~2.3.1]{BrinGuzman} there is a PL homeomorphism $\mu$
between $[0, \infty)$
and $[0, n-1)$  inducing, by conjugation, an isomorphism $\mu^*$
between $F_{n, 0}$ and a subgroup  of $PL([0, n-1])$.
Since $t_0 \notin \BZ [{1 \over n}], $ we cannot directly apply  the
  isomorphism $\mu^*$ above. Since by definition $\mu$ is the identity on
$[0, n-2]$ and $t_0 < 1 \leq
n-2,$  the restriction of  $\mu$ to $[t_0, \infty)$ induces, via
conjugation, an isomorphism between
$F_{n, [t_0, \infty]}$ and a subgroup of $PL([t_0,n-1])$.
The behaviour at $+ \infty$ is explained in the second paragraph of the
proof of
\cite[Lemma~2.3.1]{BrinGuzman}: if $f \in F_{n, [t_0, \infty]}$ is
such that $(x)f = x + (n-1)i$ for sufficiently large $x$ then the image of
$f$ in
$PL([t_0,n-1])$ has derivative at $n-1$ equal to $n^{-i}$.
\end{proof}

\begin{corollary} \label{sigma200} If (\ref{condition}) holds then
$\Sigma^1( F_{n, [t_0,
\infty]})^c = \{ [\mu_1], [\mu_2] \}$.
\end{corollary}

\begin{proof} Follows directly from Theorem \ref{sigma1}, Lemma
\ref{nonfix}, Lemma
\ref{independent} and Lemma \ref{interval}.
\end{proof}

\noindent
Let $\nu_1, \nu_2 : F_{n, [-\infty, t_0]} \to \BR$ be  characters
defined as follows:
$$\nu_1(f) = log_n ((t_0)f')$$
and
$$\nu_2(f) = - i \mbox{ where }  (x)f = x + i (n-1) \mbox{ for } x
< - M$$
for some positive number $M$ depending on $f$.

\noindent
The proof of the following result is the same as the proof of the previous
corollary.

\begin{corollary} \label{sigma211} If (\ref{condition}) holds then
$\Sigma^1( F_{n, [-
\infty, t_0]})^c = \{ [\nu_1], [\nu_2] \}$.
\end{corollary}

\noindent Set
$$
D = F_{n, [- \infty, t_0]} \times F_{n, [t_0,  \infty]}
$$
and consider the characters
$\tilde{\mu}_1, \tilde{\mu}_2 : D \to \BR$ that extend $\mu_1, \mu_2$
and are the zero
map on $ F_{n, [- \infty, t_0]}$. Similarly define $\tilde{\nu}_1,
\tilde{\nu}_2 : D \to
\BR$ to extend $\nu_1, \nu_2$ and are the zero map on $ F_{n, [t_0,
\infty]}$. Let
$\widetilde{\varphi}$ be the automorphism of $Homeo(\BR)$ sending $f$
to $h^{-1} f h$,
i.e. $\widetilde{\varphi}$ extends $\varphi$.

\begin{lemma}\label{autD} If (\ref{condition}) holds then
$\widetilde{\varphi}$ is an
automorphism of
$D = F_{n, [- \infty, t_0]} \times F_{n, [t_0,  \infty]}$ \end{lemma}

\begin{proof} Let $G$ be the subgroup of those elements of $D$ that
are differenciable at
$t_0$ i.e. $t_0$ is not a break point, so $G$ is the kernel of
$\tilde{\mu}_1 -
\tilde{\nu}_1$ and so $D / G \simeq \BZ$. Observe that  $$G = \{ f \in
F_{n, \infty} \mid
(t_0)f = t_0 \}.$$

\noindent
Since $(t_0) h = t_0$ the group $G$ is invariant under $\varphi$ i.e.
$\widetilde{\varphi}(G) = G$. To complete the proof it suffices to
take any element $f
\in D \setminus G$ that is a generator of $D / G \simeq \BZ$ and show that
$\widetilde{\varphi}(f) = h^{-1} f h \in D$. Then $\widetilde{\varphi}(D)
\subseteq D$ and as
$\widetilde{\varphi}^2 = id$ we get that $\widetilde{\varphi}(D) = D$.

Take $f$ that is the identity map on $(- \infty, t_0]$. Then
$\widetilde{\varphi}(f)$ is
the identity map on $[t_0, \infty)$. Note that $f$ coincides with
some $g \in F_{n,
\infty}$ on the interval $[t_0, \infty)$. Then
$\widetilde{\varphi}(f)$ and $\varphi(g)$
coincide on $(- \infty, t_0]$. Thus $\widetilde{\varphi}(f) \in D$.
\end{proof}

\begin{lemma}  \label{swaping0} If (\ref{condition}) holds then  the map
$\widetilde{\varphi}$ induces a map $\widetilde{\varphi}^* : S(D) \to
S(D)$ that swaps
$[\tilde{\mu}_i]$ with $[\tilde{\nu}_i]$ for $i = 1, 2$.
\end{lemma}

\begin{proof} By Corollary \ref{sigma200}, Corollary \ref{sigma211} and
the direct product
formula in dimension 1, see Proposition \ref{direct1},  we have that
$\Sigma^1(D, \BZ)^c
$ is exactly the set $\{  \tilde{\mu}_i, \tilde{\nu}_i \}_{i = 1,2}$.
Any automorphism
of $D$ induces a map  that permutes the elements of $\Sigma^1(D, \BZ)^c $.

Since $\varphi$ is given by conjugation with a decreasing $h,$ we get that
$\widetilde{\varphi}^*$ permutes the characters  corresponding to $\pm
\infty$ i.e.
$\widetilde{\varphi}^*([\nu_2]) = [\mu_2]$ and $\widetilde{\varphi}^*$
permutes the
characters  corresponding to derivatives (right or left) at $t_0$ i.e
$\widetilde{\varphi}^*([\nu_1]) = [\mu_1]$ .
\end{proof}

\begin{corollary} \label{swaping} If (\ref{condition}) holds then
the map $\widetilde{\varphi}$ induces a map
$$\widetilde{\varphi}^{**}
: Hom(D, \BR) \to
Hom (D, \BR)$$
swapping $\tilde{\mu}_i$ with $\tilde{\nu}_i$ for $i = 1, 2$.
\end{corollary}

\begin{proof} By Lemma \ref{swaping0} there are some positive real
numbers $r_1$ and
$r_2$ such that
    $\widetilde{\varphi}^{**}(\tilde{\mu}_1) = r_1 \tilde{\nu}_1$ and
$\widetilde{\varphi}^{**}(\tilde{\mu}_2) = r_2 \tilde{\nu}_2$.
    Then for every $f \in F_{n, [- \infty, t_0]}$ we have that $(t_0)
(h^{-1} f h)' =
\widetilde{\varphi}^{**}(\tilde{\mu}_1) (f) = r_1 \tilde{\nu}_1 (f) =
r_1 ((t_0) f') $
where derivatives mean right or left derivatives. Applying this to
the identity map $f$
we deduce that $r_1 = 1$.
Similarly $ r_2 = 1$.
    \end{proof}
    
The following corollary is not needed to establish the main result of
this section i.e.
Corollary
    \ref{importantcor}, but we include it as it follows easily from the
results on
$\Sigma$-theory developed in \cite{BieriRenz}.
\begin{corollary} If (\ref{condition}) holds then the group $G$ defined in
the proof of Lemma \ref{autD} is of type
$\FP_{\infty}$.
    \end{corollary}
    \begin{proof}
By  Proposition \ref{HNNext} and Proposition \ref{meinertnew} $$- [\mu_1]
\in
\Sigma^{\infty} (F_{n,[t_0, \infty]}, \BZ)$$  and similarly $$- [\nu_1]
\in
\Sigma^{\infty}
(F_{n,[- \infty, t_0]}, \BZ).$$ Then by Corollary \ref{meinert2} the
images of both
$\tilde{\mu}_1 - \tilde{\nu}_1 = (- \nu_1,  \mu_1) \in Hom (D, \BR)$
and $- \tilde{\mu}_1
+ \tilde{\nu}_1 = ( \nu_1, - \mu_1) \in Hom (D, \BR)$ represent elements
of
$\Sigma^{\infty} (D, \BZ)$ i.e.
$ [(- \nu_1,  \mu_1)], [( \nu_1, - \mu_1)] \in \Sigma^{\infty} (D, \BZ)$.
Then by \cite[Thm.~B]{BieriRenz} $G = Ker( \tilde{\mu}_1 -
\tilde{\nu}_1)$ is of type
$\FP_{\infty}$.
\end{proof}
\begin{theorem} \label{finalcentral} If (\ref{condition}) holds then
$C_{F_{n,
\infty}}(\varphi) \simeq F_{n,[t_0, \infty]}$.
\end{theorem}

\begin{proof}
Observe that
$$C_{F_{n, \infty}}(\varphi) = \{ f \in F_{n, \infty} | \varphi(f) = f
\} = \{ f \in F_{n, \infty} | hf = fh \}.$$

\noindent
Consider the morphism
$$
\rho : C_{F_{n, \infty}}(\varphi) \to F_{n, [t_0, \infty)}
$$
sending $f$ to its restriction on $[t_0, \infty)$.
\noindent Similarly to the proof of Proposition \ref{reduction1} we now
show  that $\rho$ is bijective. The fact that $\rho$ is injective is
obvious. To show that $\rho$ is surjective consider $f_1 \in F_{n, [t_0,
\infty)}$. Then there is
a unique function $\tilde{f} : \BR \to \BR$ that coincides with $f_1$
on $[t_0, \infty)$
and $\tilde{f}h = h \tilde{f}$. In fact $\tilde{f} = (f_2, f_1) \in
F_{n, [- \infty,
t_0]} \times F_{n, [t_0, \infty]}$. By Corollary \ref{swaping}
$\widetilde{\varphi}$
induces a map swapping the characters $\tilde{\mu}_1$ and that
$\tilde{\nu}_1$.
Since $\widetilde{\varphi}(f_1) = f_2,$
$t_0$ is not a break point of $f$ and so $f \in F_{n, \infty}$.
Thus $f \in C_{F_{n, \infty}}(\varphi)$ and $\rho(f) = f_1$ i.e. $\rho$ is
surjective.
\end{proof}

\noindent Theorem \ref{finalcentral} together with Proposition
\ref{HNNext} yields the desired direction of Theorem A, part (2).

\begin{corollary} \label{importantcor} If (\ref{condition}) holds then
$C_{F_{n,
\infty}}(\varphi)$ is of type $\FFF_{\infty}$.
\end{corollary}

\section{Proof of Proposition B}\label{theotherdirection}

\noindent We begin by completing the proof of part (2) of Theorem A. This
follows from combining the results of the previous section with the
following lemma.

\begin{lemma}  \label{finitegenerationused} Suppose that $C_{F_{n,
\infty}}(\varphi)$ is
finitely generated. Then
there exists an element $f \in C_{F_{n, \infty}}(\varphi)$ with slope
at $t_0$ not equal to one.
\end{lemma}

\begin{proof}
The case when $t_0 \in  \BZ[ {1 \over n}]$ is obvious, so we can
assume $t_0 \notin  \BZ[
{1 \over n}]$.
Suppose that  for every $f \in C_{F_{n, \infty}}(\varphi) $ the slope at
$t_0$
is 1.
Recall that by Lemma \ref{fixed} $(t_0) f = t_0$.  As  $C_{F_{n,
\infty}}(\varphi)$ is
finitely generated there is a small closed interval $J$ containing
$t_0$ such that the
restriction of any element of $C_{F_{n, \infty}}(\varphi)$ to $J$ is the
identity.  In fact, $J$ is  the intersection of the intervals defined for
the
finitely many generators $f$ of  $C_{F_{n, \infty}}(\varphi)$. Thus $J$ is
a closed interval, which is not a point.

\noindent
Let $r_0$ be an element of $\BZ[{1 \over n}]$ such that $r_0 > t_0$
and $r_0$ is in the interior of the interval $J$. Set $f_0 \in F_{n, r_0}$
such that the
right-hand  derivative at $r_0$ is not 1 and denote by $f_1$  the
restriction of
$f_0 $ on $[t_0,\infty)$. Thus $f_1 \in F_{n, [t_0,  \infty]}$ and  $f_2 =
h^{-1}
f_1 h$ can be thought of as an element of $F_{n, [- \infty, t_0]}$.
Since $\varphi$ has order 2,  $(f_2, f_1) \in F_{n, [- \infty,t_0]} \times
F_{n,[t_0,  \infty]}$ is a fixed point of $\varphi$ whose restriction on
$J$ is not the identity, giving
a contradiction.
\end{proof}

\noindent The following result yields Proposition B.

\begin{corollary} \label{propB} The group $C_{F_{n, \infty}}(\varphi)$
is either
infinitely generated or of type $\FFF_{\infty}$.
\end{corollary}

\begin{proof} This follows directly from Corollary \ref{importantcor}
and the previous
lemma. We observe that there is a shorter proof of this corollary
avoiding Corollary
\ref{importantcor}. If $C_{F_{n, \infty}}(\varphi)$ is finitely
generated  the previous lemma implies that there is an element $\hat{f}$
of $C_{F_{n,
\infty}}(\varphi)$ which has right-hand
slope at $t_0$ different from 1. Note that this is stronger than condition
(\ref{condition}), which requires such an element to just lie in $F_{n,
\infty}$. This
element can be used directly  in the proof of Theorem \ref{finalcentral}
to show that the
image of the map $\rho$ of the
proof of Theorem \ref{finalcentral} has finite index in $F_{n, [t_0,
\infty]}$ instead of using Corollary \ref{swaping}. Indeed, define $M$ as
the subgroup of $F_{n,
[t_0, \infty]}$ containing the elements which have slope 1 at $t_0$.
Then with the notation used in the proof of Corollary \ref{importantcor},
$M \cup \{
\rho(\hat{f}) \} \subseteq Im (\rho)$ and $\rho(\hat{f}) \notin M$. Thus
$M$ is a
proper subgroup of $Im
(\rho)$. Hence $1 \not= Im (\rho) / M \leq F_{n, [t_0, \infty]} / M \simeq
\BZ$ and $ Im (\rho)$ has finite index in $F_{n, [t_0, \infty]}$. In
particular, as $\rho$ is injective, $Im(\rho) \simeq C_{F_{n,
\infty}}(\varphi)$ is of type $\FFF_{\infty}$ as required.
\end{proof}

\begin{corollary} \label{equivalent} There exists an element $f \in
C_{F_{n,
\infty}}(\varphi)$ with right-hand slope at $t_0$ different from  1
if and only if there exists an element $f \in F_{n, \infty}$ with
right-hand slope at $t_0$
different from  1.
\end{corollary}

\begin{proof}
This follows directly from Corollary \ref{importantcor} and Lemma
\ref{finitegenerationused}.
\end{proof}

\section{Auxiliary results}

\noindent
As in \cite{BrinGuzman} we consider the function
$$
\phi_n : \BZ [{1 \over n}] \to \BZ_{n-1}
$$
given by $$\phi_n(n^a b ) \equiv b \ \ (\hbox{mod  } n-1), $$ where
$a,b \in \BZ$ and $b$
is not divisible by $n$.

\begin{lemma} \label{odd1} Let $n,k$ be coprime natural numbers. Consider
$$t_1:={m\over k}n^a,\quad t_2:={r\over k}n^c$$
such that
$$k(n-1)\mid rn^c-mn^a$$
(here, divisibility is understood in $\BZ[{1\over n}]$).
Then there exists some $g \in F_{n, \infty}$ such that
$$(t_1)g=t_2.$$
\end{lemma}

\begin{proof}Let $0<s<k$ such that $1+ns=kt$ for some $t\in\BZ$. We
assume $t_1\geq 0,$ the other case is analogous. Put
$$x_1:=t_1+t_1ns,$$
$$x_2:=t_1+t_1n(s-k),$$
$$y_1:=t_2+t_1ns,$$
$$y_2:=t_2+t_1n(s-k).$$
Clearly, $x_2\leq t_1\leq x_1$ and $y_2\leq t_2\leq y_1$. Moreover
$x_1=t_1kt\in\BZ[{1\over n}]$ and as
$x_1-x_2=y_1-y_2=t_1nk=mn^{a+1}\in\BZ[{1\over n}]$ and
$y_1-x_1=t_2-t_1={rn^c-mn^{a}\over k}\in\BZ[{1\over n}]$ we observe
that $x_1,x_2,y_1,y_2\in\BZ[{1\over n}]$. Also,
$$\phi_n(y_1-x_1)=\phi_n(y_2-x_2)=\phi_n({rn^c-mn^a\over k})=0$$
by hypothesis which implies $\phi_n(x_i)=\phi_n(y_i)$ for $i=1,2$.
Therefore by
\cite[Lemma~1.2.1]{BrinGuzman} there is $g \in F_{n, \infty}$ such
that $(x_i)g=y_i$. Moreover this $g$ in the interval $[x_2,x_1]$ is
constructed as follows: as the length of the interval is
$x_1-x_2=mn^{a+1}$, we may subdivide it in $m$ intervals of length
$n^{a+1}$ each. We do the same with $[y_2,y_1]$, which has the same
length $mn^{a+1}$ and then define $g$ mapping each subdivision of
$[x_2,x_1]$ to each subdivision of $[y_2,y_1]$. Clearly, this implies
that on $[x_2,x_1]$, $g$ is precisely the function $(x)g=x+t_2-t_1$.
Therefore
$$(t_1)g=t_2.$$
\end{proof}

\begin{lemma}\label{tec1} For any $t_0\in\Q$ there exist integers
$k,s,t>0$
such that
$n-1\mid k$ and
$$t_0={k\over n^t(n^s-1)}$$
\end{lemma}
\begin{proof} Put $t_0=r/m$ with $n-1\mid r$ (we don't require $r,m$
coprime) and
decompose $m=m_1m_2$ so that $m_2,n$ are coprime and all the primes
dividing $m_1$
divide also $n$. Then there is some $t>0$ such that $m_1\mid n^t$ and
some $s>0$ such
that $n^s\equiv 1$ mod $m_2$. Therefore $m=m_1m_2\mid n^t(n^s-1)$ and
for some $r_1$ we
have $r_1 m=n^t(n^s-1)$. Putting $k=r_1r$ we are done.
\end{proof}

\begin{lemma}\label{odd2} Let $t_0\in\R$. Then there is some $f\in
F_{n,\infty}$ such
that $(t_0)f=t_0$ and the right-hand slope of $f$ at $t_0$ is not one if
and only if
$t_0\in\Q$.
\end{lemma}
\begin{proof} Assume first that there exists such an element $f$. Then for
some interval $I$ with
$t_0\in I$, $(x)f=ax+b$ for any $x\in I$ where $a,b\in\Q$ and $a\neq
1$. Then the
equation $(t_0)f=t_0$ implies $t_0\in\Q$.

Next, we assume that $t_0\in\Q$ so by Lemma \ref{tec1} we may put
$t_0={k\over
n^t(n^s-1)}$. Let $a=n^s= \langle n \rangle $, $b=-{k\over
n^t}\in\Z[1/n]$. Then for
any function $f$
such that in a neighbourhood of $t_0$,  $(x)f=ax+b$ we have
$(t_0)f=t_0$. We claim that there is
some such function  in $F_{n,\infty}$. To prove the claim  it is
sufficient to find
$x_1,x_2\in\Z[1/n]$ such that $t_0\in[x_1,x_2]$ and  for $y_i=ax_i+b$,
$i=1,2$
$$y_i-x_i=l_i(n-1)$$
for some integers $l_i.$ Then the function
$$(x)f=\Bigg\{\begin{aligned}
&x+l_1(n-1)\text{ if }x\leq x_1\\
&ax+b\text{ if }x_1<x\leq x_2\\
&x+l_2(n-1)\text{ if }x_2<x\\
\end{aligned}$$
  is in $F_{n,\infty}$ and satisfies the conditions we wanted. So we
need (recall that
$n-1\mid k$ and denote $k_1=k/(n-1)$)
$$y_i-x_i=(a-1)x_i+b=(n^s-1)x_i-k/n^t=l_i(n-1).$$
Restricting to $x_i$ of the shape $\alpha_i/n^t$ with $\alpha_i\in\Z$
this yields
$${(n^{s-1}+\ldots+n+1)\alpha_i-k_1\over n^t}=l_i\in\Z.$$
As $n^{s-1}+\ldots+n+1$ is coprime to $n^t$ we may choose
infinitely many values of
$\alpha_i$ having this property. Note also that $\alpha_1$ can be
chosen arbitrarily
small and $\alpha_2$ arbitrarily big so that $t_0\in[x_1,x_2]$.
\end{proof}

\section{Conjugacy classes} \label{sectionconjugacy}

\noindent
{\bf Definition.}  {\it Let $\varphi$ be an automorphism of $F_{n,
\infty}$ of order 2
given by conjugation with a homeomorphism  $h$ of the real line. Let
$k>0$ be an integer. We say that $h$ has
property $(\ast)$ for $k$ if for every  element $h_i$ of order 2 in $F_{n,
\infty} \rtimes \langle h
\rangle $  the unique $t_i \in \BR$ such that $(t_i) h_i = t_i$
belongs to ${1 \over
k}\BZ[{1 \over n}]$.}

  \medskip\noindent
  Note that for $k=k_1p$ with $p,n/p \in \mathbb{Z}$, we have
    $${1\over k}\BZ[{1\over n}]\subseteq{1\over k_1}\BZ[{1\over n}].$$
This implies that if $h$ has property $(\ast)$ for some $k$, then it also
has property $(\ast)$ for some $k'$ which is coprime to $n$.
\medskip

\noindent By Lemma \ref{odd2}  condition (\ref{condition}) of Section
\ref{sectiondirection}  holds for the order 2 automorphism induced by
conjugation with
each $h_i$  satisfying the above definition.

\medskip

\begin{lemma}\label{tec3}
Let $\varphi$ be an automorphism of $F_{n, \infty}$ of order 2
given by conjugation with a homeomorphism  $h$ of the real line having
property $(\ast)$ for some $k$.
Let $f,\widetilde{f}\in F_{n, \infty}$ such
that $\varphi(f) =
f^{-1}$, $\varphi(\widetilde{f}) = \widetilde{f}^{-1}$ and $(t_0)f =
(t_0) \widetilde{f}
= t_0 = (t_0)h$. Then there is an element $g \in F_{n, \infty}$ such that
$
{\varphi}(g)^{-1} f g =
\widetilde{f}$ and $(t_0)g = t_0$.
\end{lemma}

\begin{proof} Note that $ f h$ and
    $ \widetilde{f} h$ are  elements of order 2 in $F_{n, \infty} \rtimes
\langle h \rangle$ and that
$$\widetilde{\varphi}(f_1) = f_2^{-1}$$ and $$
\widetilde{\varphi}(\widetilde{f}_1) = \widetilde{f}_2^{-1},$$
where $f = (f_1, f_2)$, $\widetilde{f} = (\widetilde{f}_1,
\widetilde{f}_2) \in D = F_{n, (-\infty, t_0]} \times F_{n, [t_0,
\infty)}$ and $\widetilde{\varphi} \in Aut(D)$ is given by
conjugation with $h$. The automorphism $\widetilde{\varphi}$ is
well-defined by Lemma \ref{autD} as $(t_0) h = t_0$.

\noindent We shall construct an element $g = (g_1, g_2) \in F_{n, (-
\infty, t_0]} \times
F_{n, [t_0, \infty)}$ with $g\in F_{n,\infty}$ and such that
    $$
    \widetilde{\varphi}(g_2)^{-1} f_1 g_1 = \widetilde{f}_1$$ and
$$\widetilde{\varphi}(g_1)^{-1}
f_2 g_2 = \widetilde{f}_2.
    $$
It is sufficient that the first of the above equalities holds. In case
$t_0 \in \BZ[{1 \over n}]$, in particular $t_0$ can be a break point of
$g$,
  we define $g_2 = id$ and
$g_1 = f_1^{-1} \widetilde{f}_1$.

\noindent
Suppose now that  $t_0 \notin \BZ[{1 \over n}]$ and hence $t_0$ cannot  be
a break point of
$g$.
  We can define $g_1$ as $f_1^{-1} \widetilde{\varphi}(g_2)
\widetilde{f}_1$ but have to
be sure that $g_2$ is chosen in such a way that  $t_0$ is not a break
point of $g$ i.e.
$\nu_1(g_1) = \mu_1(g_2)$.
  Note that by Corollary \ref{swaping},  $\nu_1
\widetilde{\varphi}=\widetilde{\varphi}^*(\nu_1)=\mu_1$.  Hence
    $$\begin{array}{ll}
    \nu_1(g_1)  & = \nu_1(f_1^{-1} \widetilde{\varphi}(g_2)\widetilde{f}_1)
\\
    & = \nu_1(f_1^{-1})+ \nu_1( \widetilde{\varphi}(g_2)) + \nu_1(
\widetilde{f}_1) \\
    & = - \nu_1(f_1) + \nu_1(\widetilde{f}_1) + \mu_1(g_2).\end{array}$$
    Furthermore,  since $f_1 = \widetilde{\varphi}(f_2)^{-1}$ we have
    $$\nu_1(f_1) = \nu_1 (\widetilde{\varphi}(f_2)^{-1}) =  -
\nu_1(\widetilde{\varphi}(f_2)) = - \mu_1 (f_2).$$
But as $t_0$ is not a break point of $f$, we have $\mu_1(f_2) =
\nu_1(f_1)$.
Therefore $\nu_1(f_1) = 0$, and similarly $\nu_1( \widetilde{f}_1) = 0.$
    Combining this with the calculations above we obtain
    $$
    \nu_1(g_1) = - \nu_1(f_1) + \nu_1( \widetilde{f}_1) + \mu_1(g_2) =
\mu_1(g_2).
    $$
\end{proof}

\begin{lemma}\label{tec4} Let $h_1$ and $ h_2$ be elements of order 2 in
$G =
F_{n, \infty} \rtimes \langle h \rangle$, where $h$ has property $(\ast)$
for some $k$.  Let $t_i$ be the
unique real number such that $(t_i)h_i = t_i$ for $i = 1,2$. If there
is some $\hat g\in
F_{n,\infty}$ such that $(t_1)\hat g=t_2,$ then $h_1$ and $h_2$ are
conjugated in $G$.
\end{lemma}
\begin{proof}
Consider $f \in F_{n, \infty}$ defined by $$h_2 = f h_1.$$
We
write $G$ as $F_{n, \infty} \rtimes \langle h_1 \rangle $ and $\varphi_1
\in Aut(F_{n,
\infty})$ is conjugation by $h_1$. Let $\widetilde{f} = \hat{g} f
\varphi_1(\hat{g})^{-1}$. Then
$$
(t_1)\widetilde{f} =(t_1) \hat{g} f \varphi_1(\hat{g})^{-1}=(t_2)f
h_1\hat{g}^{-1}h_1=(t_2)h_2 \hat{g}^{-1}h_1 =
(t_2) \hat{g}^{-1}h_1=
t_1.
$$
Note that since $h_2^2 = id$ we have $\varphi_1(f) = f^{-1}$, hence
$\varphi_1 (\widetilde{f}) = \widetilde{f}^{-1}$.
Now by Lemma \ref{tec3} there is a $g \in F_{n, \infty}$ such that
$g\widetilde{f}\varphi_1({g})^{-1} = id$. Thus for $\widetilde{g} = g
\hat{g}$ we have
$$\begin{aligned}
id =g\widetilde{f}\varphi_1({g})^{-1}=g\hat{g} f
\varphi_1(\hat{g})^{-1}\varphi_1({g})^{-1}
    = & \widetilde{g}  f \varphi_1(\widetilde{g})^{-1}=\\
     \widetilde{g}  h_2h_1 h_1(\widetilde{g})^{-1} h_1  =
& \widetilde{g} h_2  (\widetilde{g})^{-1} h_1,
\end{aligned}$$
hence
$$
h_1 = h_1^{-1}  = \widetilde{g}^{-1} h_2 \widetilde{g}.
$$
\end{proof}

\begin{theorem} \label{thmclasses} Let $\varphi$ be an automorphism of
$F_{n, \infty}$ of
order 2 given by conjugation with a homeomorphism  $h$ of the real line.
Let  $t_0$ be the unique real number such that $(t_0)h = t_0$.
\begin{itemize}
\item[a)] If $h$ is PL, then  $h$ has property $(\ast)$ for 2. In
particular $t_0  \in {1 \over
2}\BZ[{1 \over n}]$;
\item[b)] If $h$ has property $(\ast)$ for some integer, then  there are
only
finitely many conjugacy
classes of elements of order 2 in $F_{n, \infty} \rtimes \langle
h\rangle$.
\end{itemize}

\end{theorem}
\begin{proof}
     a) Let $f \in F_{n, \infty}$ be such that $fh $ has order 2 i.e.
$\varphi(f) = f^{-1}$ and $r_0$ be the unique real number such that
$(r_0)fh = r_0$.
  Suppose that $r_0 \notin \BZ[{1 \over n}]$, otherwise there is
nothing to prove. Then
for some small $\epsilon > 0$ we  have that the restriction of $f$ on
$[r_0 - \epsilon,
r_0 + \epsilon]$ is a linear function, say $n^s t + \lambda$ for some
$\lambda \in \BZ[{1
\over n}]$, $s \in \BZ$. Since $h$ is PL, on a neighborhood of
$(r_0)h$ we have
$(t)h = a_1 t + b_1$, where the slope is in the multiplicative group
of $\BR$ generated
by the prime divisors of $n$ and $(\BZ[{1 \over n}])h = \BZ[{1 \over n}]$
\cite[Lemma~3.2.2]{BrinGuzman}.

\noindent   Since $(x)hfhf = x$ and $(r_0)h^2 f = (r_0)f = (r_0)h$, we
deduce
that for $x$ in a small
neighbourhood of $(r_0)h$ we have
$$
x = (x)hfhf = n^s( a_1(n^s(a_1 x + b_1)  + \lambda))+ b_1) + \lambda
=$$ $$ (n^s a_1)^2 x
+
(n^s)^2 a_1 b_1 + n^s a_1 \lambda + n^s b_1 + \lambda.
$$
Then $(n^s a_1)^2 = 1$, hence $n^s a_1 = -1$ since $a_1 < 0$ (remember
$h$ is decreasing)
i.e. $a_1 = - n^{-s}$. Since $(\BZ[{1 \over n}])h = \BZ[{1 \over n}]$
we deduce that $b_1
\in \BZ[{1 \over n}]$. Then
$$
r_0 = (r_0)h^2 = a_1((r_0)h) + b_1
$$
and
$$
(r_0)h = (r_0)f = n^s r_0 + \lambda = (r_0)hhf = n^s(a_1 ((r_0)h) + b_1) +
\lambda
= $$ $$n^s a_1 ((r_0)h) + n^s b_1 + \lambda = - (r_0)h + n^s b_1 +
\lambda.
$$
Thus
$$(r_0)f = (r_0)h = {1 \over 2} (n^s b_1 + \lambda) \in {1 \over 2}
\BZ[{1 \over n}].
$$
Since $r_0 \notin \BZ[{1 \over n}]$ we have $(r_0)f \notin \BZ[{1
\over n}]$ and $(r_0)f$
is not a break point of $f^{-1}$. So in a neighborhood of $(r_0)f$ the
function $f^{-1}$
is given by $(x)f^{-1} = n ^{\alpha} x + \beta$, where $\beta \in
\BZ[{1 \over n}],
\alpha \in \BZ$. Thus
$$
r_0 = (r_0) f f^{-1} = n^{\alpha} ((r_0)f) + \beta  \in {1 \over 2}
\BZ[{1 \over n}]
$$
as required.
Note that for $n$ even we have ${1 \over 2} \BZ[{1 \over n}] \subseteq
\BZ[{1 \over n}]$.
Thus
$$t_0, r_0 \in \BZ[{1 \over n}] \hbox{ if } n \hbox{ is even.}$$

\medskip

b) Let $k$ be an integer such that $h$ has $(\ast)$ for $k$. By the
comment
after the definition of propery $(\ast)$ we may assume that $n,k$ are
coprime. Let $h_1, h_2$ be elements of order 2 in $G = F_{n, \infty}
\rtimes \langle h
\rangle$. Let $t_i$ be the
unique real number such that $(t_i)h_i = t_i$ for $i = 1,2$. Put
$$t_1={m\over k}n^a,\quad t_2={r\over k}n^c$$
for certain integers $m,r,a,c$. Then if these integers satisfy
$$k(n-1)\mid kt_2-kt_1= rn^c-mn^a$$
in $\BZ[{1\over n}]$, by Lemma \ref{odd1}
there is some $\hat g\in F_{n,\infty}$ such that $(t_1)\hat g=t_2$.
Therefore $h_1$ and $h_2$ are conjugated by Lemma \ref{tec4}. As the
number
of possible values of $kt_i$ in the quotient ring $\BZ[{1\over
n}]/k(n-1)\BZ[{1\over n}]$ is finite, the result follows.

    \end{proof}
\section{Proof of Theorem A} \label{sectionproofA}

  \noindent  The first part of Theorem A is Lemma \ref{order2}.
The second part of Theorem A follows from  Corollary
\ref{importantcor} and Lemma
\ref{finitegenerationused}.

\noindent  To show the third part of Theorem A  we use  that by Theorem
\ref{thmclasses}, $t_0 \in {1 \over
2} \BZ[{1 \over n}]$. Note that if $n$ is even,  $t_0 \in \BZ[{1 \over
n}]$ and if $n$ is
odd we can apply Lemma \ref{odd2}. In both cases  there is  an $f \in
F_{n, \infty}$ such
that $(t_0)f = t_0$ and the right-hand slope of $f$ at $t_0$ is not 1.
Then
part (2) of Theorem A implies that  $C_{F_{n, \infty}}(\varphi)$ is of
type $\FFF_{\infty}$.

    \section{Proof of Theorem C and Corollary D} \label{sectionproofCD}

\noindent {\bf Theorem C.} {\it Any finite extension $G$ of $F_{n,
\infty}$ where
the elements of
finite order act on $F_{n, \infty}$ via conjugation with PL
homeomorphisms  is of type Bredon
$\FFF_{\infty}$.
  }

    \begin{proof} By Lemma \ref{order2} any element of finite order in
$Aut(F_{n, \infty})$
has order 2 and is given by conjugation with a decreasing
homeomorphism of the real line
(of multiplicative order 2).
    If a product of such two elements has finite order then the product
is a conjugation by
an increasing function ( i.e. a composition of two decreasing function
is increasing), so
cannot have a finite order unless the product is the trivial element.
Thus every finite
subgroup $B$ of $G$ has a subgroup $B_0$ of index at most two  acting
trivially on
$F_{n, \infty}$.

    Let $\pi : G \to G / F_{n, \infty}$ be the canonical projection  and
let $P$ and $Q$ be
finite subgroups of $G$ such that $\pi(P) = \pi(Q)$. Since $F_{n,
\infty}$ does not have
non-trivial elements of finite order $F_{n, \infty} \cap P = 1 = F_{n,
\infty} \cap Q$,
so we consider $F_{n,\infty} \rtimes P = F_{n,\infty}  \rtimes Q $ as
a subgroup of $G$.
Since there are only finitely many possibilities for $\pi(P)$ it
suffices to show that
$F_{n,\infty}  \rtimes P$ has finitely many conjugacy classes of
finite subgroups, so
without loss of generality we can assume that $G = F_{n,\infty}  \rtimes
P$.

\noindent
Let $P_0$ be the subgroup of $P$ of elements acting trivially on
$F_{n, \infty}$ via
conjugation.
    Consider the short exact sequence of groups $P_0 \mono G \epi F_{n,
\infty} \rtimes (P /
P_0)$. $P / P_0$ is either trivial or a cyclic group of order 2. Hence 
Theorem \ref{thmclasses} implies that there are only finitely many conjugacy classes of
finite subgroups in $F_{n, \infty} \rtimes (P / P_0)$. Since $P_0$ is finite, this
implies that there are only finitely many conjugacy classes of finite subgroups in $G$.

Finally  we show that the centraliser $C_G(P)$ of $P$ in $G$ is of
type $\FFF_{\infty}$. If
$P = P_0$ then $C_G(P)$ is a finite extension of $F_{n, \infty}$, so
is of type
$\FFF_{\infty}$. If $P \not= P_0$ let the action of $P / P_0$ on $F_{n,
\infty}$ be given by
conjugation with a real homeomorphism $h$ and
    define $t_0$ as the unique real number such that $(t_0)h = t_0$. By
Theorem
\ref{thmclasses}, $t_0 \in {1 \over 2} \BZ[{1 \over n}]$ as the
hypothesis imply that $h$
is PL. Note that by Lemma \ref{odd2}
there is $f \in F_{n, \infty}$ such that $(t_0)f = t_0$ and the right-hand
slope of $f$ at
$t_0$ is not 1. Then for $\varphi \in Aut(F_{n, \infty}) $ given by
conjugation with $h$
by Theorem A, part(2)
$$C_G(P) \cap F_{n, \infty} = C_{F_{n, \infty}} (P) = C_{F_{n, \infty}}
(\varphi)$$
    is of type $\FFF_{\infty}$. Note that $C_{F_{n, \infty}} (\varphi)$ is
a subgroup
of
finite index in $C_G(P)$, hence $C_G(P)$ is of type $\FFF_{\infty}$ as
required.
    \end{proof}

\noindent{\bf Corollary D. }{\it Any finite extension of $F$ is of type
Bredon $\FFF_{\infty}$.}

     \begin{proof} By \cite[Lemma~5.1]{Brin} for $n = 2$ any orientation
preserving
automorphism $\theta$ is given by conjugation with a PL homeomorphism
$g$ of the real
line. As an automorphism of finite order is the conjugation with a
homeomorphism of the
real line $h$ that is the composition  of some $g$ as above with the map
$h_0 :
x \to - x$ we
deduce that $h$ is PL. Now we can apply the previous theorem.
     \end{proof}

\section{Exotic automorphisms}

\noindent
It was shown in  \cite[Thm.~6.1.5]{BrinGuzman} that for each $n \geq 3$
there exist \lq\lq exotic"
orientation preserving automorphisms $\theta$ of $F_{n, \infty}$, where
exotic means that the
automorphisms are given by
conjugation with a homeomorphism of the real line, which is not PL.

In this section  we shall show that for each $n \geq 3$ there exist
infinitely many exotic automorphisms $\varphi$ of order $2$ in $Aut(F_{n,
\infty})$. It will turn out, however, that the unique element fixed by
these automorphisms as constructed in Lemma \ref{order2} is $t_0=0.$ Then
by Theorem A, part(2) $C_{F_{n,\infty}}(\varphi)$ is of type
$\FFF_{\infty}.$

We adopt the notation of \cite[Section 2.1, Section 4]{BrinGuzman} and
following \cite{BrinGuzman} all homomorphisms, in particular
automorphisms, considered in this section act on the right. Consider the
PL functions $g_{n,i}$ and $t_i$ given by
$$
(x)g_{n,i}=\Bigg\{\begin{aligned}
&x\quad x<i\\
&n(x-i)+i\quad i\leq x\leq i+1\\
&x+n-1\quad x>i+1\\
\end{aligned}
$$
and $t_s: x \mapsto x +s$ (here $n,s\in\BZ$). Please note that $t_i$ is
not to be confused with the element $t_0 \in \R$ of Lemma \ref{order2}.
Then by \cite[Lemma 2.1.2]{BrinGuzman}
$$F_{n,\infty}=\langle t_{n-1},g_0,\ldots,g_{n-2} \rangle$$
and $t_s^{-1}g_{n,i}t_s=g_{n,i+s}$.

Let $Aut^0(F_{n,\infty}, t_1)$ denote  the set of automorphisms of $F_{n,
\infty}$ given by conjugation with a
homeomorphism of the real line fixing  0 and commuting with $t_1$.
For $n \geq 3$ we consider the map constructed in
\cite[Thm.~4.3.1]{BrinGuzman}
    $$
    \Lambda_{2,n} : Aut^0 (F_{2, \infty}, t_1) \to Aut^0 (F_{n, \infty},
t_1),
    $$ for which the image of every  non-trivial element is exotic.

\noindent
Let $\rho_j$ be the automorphism of $F_{j, \infty}$ given by  conjugation
with $h_0 : x \mapsto - x$ and choose an arbitrary element $\theta \in
Aut^0 (F_{2, \infty}, t_1)$ such that $ (\theta\rho_2)^2 = id$ (for the
existence of a suitable $\theta$, see the end of  page 17). Denote by
$\theta_n = \Lambda_{2,n}(\theta)$. We now show that $ (\theta_n\rho_n)^2
= id,$ hence $\theta_n\rho_n \in Aut^0 (F_{n, \infty}, t_1)$ is an exotic
element of finite order.

  A simple calculation shows that for $n \geq 2$
$$(g_{n,i})\rho_n = g_{n,-i-1} t_{n-1}^{-1}.$$
Write $(g_{2,0})\theta = w(g_{2,0}, t_1)$ as a word in the generators
$g_{2,0}$ and $t_1$.
For all $0\leq i <n$ the proof of \cite[4.3.1]{BrinGuzman} together with
\cite[4.1.2]{BrinGuzman} imply
that
$$(g_{n,i})\theta_n=w(g_{n,i},t_{n-1}).$$
Note also that for $n \geq 2$
$$(t_i)\rho_n=t_i^{-1}$$
and since $t_{n-1} \in \langle t_1 \rangle$, for $n \geq 3$ we have
$$(t_{n-1})\theta_n=t_{n-1}.$$

\noindent Furthemore
$$(g_{2,0})\theta\rho_2 = ( w(g_{2,0}, t_1))\rho_2 = w((g_{2,0})\rho_2,
(t_1)\rho_2) = w(g_{2,-1} t_1^{-1} , t_1^{-1}),$$
and similarly for every $i\geq 0$ and $n \geq 3$
$$g_{n,i}(\theta_n\rho_n)=w(g_{n,i},t_{n-1})\rho_n=w(g_{n,i}\rho_n,t_{n-1}\rho_n)=w(g_{n,-i-1}t_{n-1}^{-1},t_{n-1}^{-1}).$$

\noindent{\bf Claim.} $g_{n,0}(\theta_n\rho_n)^2=g_{n,0}.$

\medskip\noindent
We observe that
\begin{align}\label{conta1}
  g_{n,0}(\theta_n\rho_n)^2 & =
w(g_{n,-1}t_{n-1}^{-1},t_{n-1}^{-1})\theta_n\rho_n \notag \\
   & = w(t_{n-1}g_{n,n-2}t_{n-1}^{-2},t_{n-1}^{-1})\theta_n\rho_n \notag \\
   & = w((t_{n-1}g_{n,n-2}t_{n-1}^{-2})\theta_n\rho_n,
(t_{n-1}^{-1})\theta_n\rho_n) \notag \\
   &= w(t_{n-1}^{-1}(g_{n,n-2})(\theta_n\rho_n) t_{n-1}^2, t_{n-1}) \\
   &=
w(t_{n-1}^{-1}w(g_{n,-(n-1)}t_{n-1}^{-1},t_{n-1}^{-1})t_{n-1}^2,t_{n-1}
)\notag \\
   & =
w(t_{n-1}^{-1}w(t_{n-1}g_{n,0}t_{n-1}^{-2},t_{n-1}^{-1})t_{n-1}^2,t_{n-1}).
\notag
\end{align}

\noindent Similarly using that $(\theta\rho_2)^2 =id$ we have
\begin{align}\label{conta2}
g_{2,0} & =g_{2,0}(\theta\rho_2)^2 \notag \\
& = w(g_{2,-1} t_1^{-1}, t_1^{-1})\theta\rho_2 \notag \\
& = w(t_1 g_{2,0} t_1^{-2}, t_1^{-1})\theta\rho_2 \notag \\
& = w( (t_1 g_{2,0} t_1^{-2})\theta\rho_2, (t_1^{-1})\theta\rho_2) \\
& = w( t_1^{-1}  (g_{2,0})(\theta\rho_2) t_1^2 , t_1) \notag \\
& = w( t_1^{-1} w(g_{2,-1} t_1^{-1}, t_1^{-1}) t_1^2 , t_1) \notag \\
& = w( t_1^{-1} w(t_1 g_{2,0} t_1^{-2}, t_1^{-1}) t_1^2 , t_1). \notag
\end{align}

\noindent Now consider  the subgroup $H$ of $F_{n, \infty}$ generated by
$g_{n,0}$ and $t_{n-1}$.
There is an isomorphism $\nu : H \to F_{2, \infty}$ given by
$g_{n, 0}\nu = g_{2,0}$ and $t_{n-1}\nu = t_1$.
Applying (\ref{conta1}) and (\ref{conta2}) yields
\begin{align*}
   (g_{n,0}(\theta_n\rho_n)^2)\nu  & = ( w(t_{n-1}^{-1} w(t_{n-1} g_{n,0}
t_{n-1}^{-2}, t_{n-1}^{-1}) t_{n-1}^{2}, t_{n-1}))\nu  \\
   & = w( t_1^{-1} w(t_1g_{2,0} t_1^{-2}, t_1^{-1}) t_1^2 , t_1)  \\
   & = g_{2,0} = g_{n,0}\nu.
   \end{align*}
Hence the claim
$$
g_{n,0}(\theta_n\rho_n)^2 = g_{n,0}.
$$
follows.
Finally, for all $i>0$ we obtain
$$g_{n,i}(\theta_n\rho_n)^2 = (g_{n,0}^{t_1^i})(\theta_n\rho_n)^2 =
g_{n,0}^{t_1^i(\theta_n\rho_n)^2} = g_{n,0}^{t_1^i} = g_{n,i},$$
which yields
$$(\theta_n\rho_n)^2 = id$$
as required.

\medskip\noindent
Recall,  that by definition $\theta$ is given by conjugation with an
orientation preserving homeomorphism $f_{\theta}$ of the real line fixing
$0$ and commuting with $t_1$. So we have for any $s\in\BZ$
\begin{equation}
\label{ftheta}
(x+s) f_{\theta}= (x)f_{\theta} + s.
\end{equation}

\noindent The condition $(\theta\rho_2)^2 = id$ is equivalent with
$(f_{\theta}h_0)^2=1$ and (\ref{ftheta}) implies that this is equivalent
with
requiring that for all $x \in [0,1]$
$$(1-(1-x)f_{\theta})f_{\theta} = x.$$
By  \cite[Lemma~5.1]{Brin}
$f_{\theta}$ is PL and furthermore, by \cite[Thm.~1,~iv)]{Brin}
its restriction $f_1$ to $[0,1]$ is in $PL_2[0,1]$.
Hence $f_1$ is an element of $F=F_{2,\infty}$ via its realization on
the unit interval.

\noindent
There are infinitely many elements $f_1 \in F$ that via the outer
automorphism given by
conjugation with $x \to 1- x$ (i.e. flipping over  the interval) are sent
to $f_1^{-1}$. Any such $f_1$ gives rise to some $\theta$ whose image
$\Lambda_{2,n}(\theta)$ is not PL  and such that
$( \Lambda_{2,n} (\theta)\rho_n)^2 = id$. Hence we have  produced
infinitely many
automorphisms of order 2 of $Aut(F_{n, \infty})$, which  are not PL. But,
as mentioned above, this way we are not going to obtain potential examples
$\varphi$ of finite order in $Aut(F_{n, \infty})$, for which
$C_{F_{n,\infty}}(\varphi)$ is not finitely generated.

\end{document}